\documentclass[11pt]{article}
\usepackage{fullpage}
\usepackage[psamsfonts]{eucal}
\usepackage{amsmath}
\usepackage{amsfonts}
\usepackage{amssymb}
\usepackage{amstext}
\usepackage{gensymb}
\usepackage{amsthm}
\usepackage{makeidx}
\usepackage{graphicx}
\usepackage{todonotes}
\usepackage[colorlinks=true,linkcolor=blue]{hyperref}
\hypersetup{
   citecolor=blue,
   linkcolor=blue,
}
\usepackage{enumerate}
\usepackage{comment}
\usepackage[capitalise]{cleveref}

\usepackage{microtype}
\usepackage{abbreviations}

\renewcommand{\P}{\mathbb{P}}
\newcommand{\T}{\tau}
\newcommand{\one}{\mathbf{1}}
\newcommand{\reff}{\mathrm{Reff}}

\newcommand{\abs}[1]{\left\vert #1\right\vert}

\newcommand{\eps}{\varepsilon}

\newcommand{\del}[1]{}  

\newcommand{\poly}{\mathrm{poly}}
\renewcommand{\O}{{O}}
\newcommand{\Ot}{\widetilde{O}}

\theoremstyle{plain}
\newtheorem{thm}{Theorem}
\newtheorem{theorem}[thm]{Theorem}
\newtheorem{prop}[thm]{Proposition}

\newtheorem{lemma}[thm]{Lemma}

\theoremstyle{definition}

\newtheorem{remark}[thm]{Remark}
\newtheorem{corollary}[thm]{Corollary}

\numberwithin{thm}{section}
\newtheorem{openprob}{Open Problem}

\begin{document}
\title{Support of Closed Walks and Second Eigenvalue Multiplicity of the Normalized Adjacency Matrix}
\author{Theo McKenzie\footnote{\texttt{mckenzie@math.berkeley.edu.} Supported by NSF Grant DGE-1752814.}\\ UC Berkeley \and Peter M. R. Rasmussen\footnote{\texttt{pmrr@di.ku.dk.} Supported in part by grant 16582, Basic Algorithms Research Copenhagen (BARC), from the VILLUM Foundation.}\\ University of Copenhagen \and Nikhil Srivastava\footnote{\texttt{nikhil@math.berkeley.edu.} Supported by NSF Grants CCF-1553751 and CCF-2009011.}\\ UC Berkeley}
\date{}
\maketitle
\begin{abstract}
    We show that the multiplicity of the second normalized adjacency matrix eigenvalue of
	any connected graph of maximum degree $\Delta$ is bounded by $O(n
	\Delta^{7/5}/\log^{1/5-o(1)}n)$ for any $\Delta$, and by
	$O(n\log^{1/2}d/\log^{1/4-o(1)}n)$  for simple $d$-regular graphs when $d\ge \log^{1/4}n$. In fact, the same bounds hold for the number of
	eigenvalues in any interval of width $\lambda_2/\log_\Delta^{1-o(1)}n$
	containing the second eigenvalue $\lambda_2$. The main ingredient in
	the proof is a polynomial (in $k$) lower bound on the typical support
	of a closed random walk of length $2k$ in any connected graph,
	which in turn relies on new lower bounds for the entries of the Perron
	eigenvector of submatrices of the normalized adjacency matrix.
\end{abstract}
\tableofcontents
\section{Introduction} 
The eigenvalues of matrices associated with graphs play an important role in many areas of mathematics and computer science, so general phenomena concerning them are of broad interest. 
In their recent beautiful work on the equiangular lines problem, Jiang, Tidor, Yao, Zhang, and Zhao \cite{equiangular} proved the following novel result constraining the distribution of the adjacency eigenvalues of {\em all} connected graphs of sufficiently low degree.
\begin{theorem}\label{thm:zhao}
If $G$ is a connected graph of maximum degree $\Delta$ on $n$ vertices, then the multiplicity of the second largest eigenvalue of its adjacency matrix $A_G$ is bounded by $O(n\log \Delta /\log\log(n)).$
\end{theorem}
For their application to equiangular lines, \cite{equiangular} only needed to show that the multiplicity of the second eigenvalue is $o(n)$, but they asked whether the $O(n/\log\log(n))$ dependence in Theorem \ref{thm:zhao} could be improved, noting a huge gap between this and the best known lower bound of $\Omega(n^{1/3})$ achieved by certain Cayley graphs of $\mathrm{PSL}(2,p)$ (see \cite[Section 4]{equiangular}). Apart from Theorem \ref{thm:zhao}, there are as far as we are aware no known sublinear upper bounds on the second eigenvalue multiplicity for any general class of graphs, even if the question is restricted to Cayley graphs (unless one imposes a restriction on the spectral gap; see Section \ref{sec:related} for a discussion).

Meanwhile, in the theoretical computer science community, the largest eigenvalues of the {\em normalized} adjacency matrix $\tilde A_G:=D^{-1/2}_GA_GD^{-1/2}_G$ (for $D_G$ the diagonal matrix of degrees) have received much attention over the past decade due to their relation with graph partitioning problems and the unique games conjecture (see e.g. \cite{kolla2011spectral,barak2011rounding, louis2012many, gharan2013new,lee2014multiway, arora2015subexponential, barak2015making, lyons2018sharp}); in particular, many algorithmic tasks become easier on graphs with few large normalized adjacency eigenvalues. Thus, it is of interest to know how many of these eigenvalues there can be in the worst case.

In this work, we prove significantly stronger upper bounds than Theorem \ref{thm:zhao} on the second eigenvalue multiplicity for the normalized adjacency matrix. {Graphs are undirected and allowed to have
multiedges and self-loops, unless specified to be simple}. Order the eigenvalues of $\tilde A_G$ as
$\lambda_1(\tilde A_G)\ge\lambda_2(\tilde A_G)\geq \ldots\geq\lambda_n(\tilde A_G)$, and let
$m_G(I)$ denote the number of eigenvalues of $\tilde A_G$ in an interval $I$.
\begin{theorem}\label{thm:multbound1}
If $G$ is a connected graph of maximum degree $\Delta$ on $n$ vertices with $\lambda_2(\tilde A_G)=\lambda_2$, then\footnote{All asymptotics are as $n\rightarrow\infty$ and the notation $\Ot(\cdot)$  suppresses $\mathrm{polyloglog}(n)$ terms.}
\begin{equation}\label{eqn:mainbound1}
  m_G\left([(1-\frac{\log\log_\Delta n}{\log_\Delta n})\lambda_2,\lambda_2]\right)=\Ot\left(n\cdot \frac{\Delta^{7/5}}{\log^{1/5} n}\right).
\end{equation}
\end{theorem}
Because of the relationship $\tilde A_G=\frac1dA_G$ when $G$ is regular, (\ref{eqn:mainbound1}) gives a substantial improvement on Theorem \ref{thm:zhao} in the regular case (in the non-regular case, the results are incomparable as they concern different matrices). In addition to the stronger $O(n/\polylog(n))$ bound, a notable difference between our result and Theorem \ref{thm:zhao} is that we control the number of eigenvalues in a small interval containing $\lambda_2$. Though we do not know whether the exponents in \eqref{eqn:mainbound1} are sharp, we show in Section \ref{sec:examples1} that constant degree bipartite Ramanujan graphs have at least $\Omega(n/\log^{3/2}n)$ eigenvalues in the interval appearing in \eqref{eqn:mainbound1}, indicating that $O(n/\polylog(n))$ is the correct regime for the maximum number of eigenvalues in such an interval when $\Delta$ is constant.

Theorem \ref{thm:multbound1} is nontrivial for all $\Delta=\widetilde{o}(\log^{1/7}n)$;  as remarked in \cite{equiangular}, Paley graphs have degree $\Omega(n)$ and second eigenvalue multiplicity $\Omega(n)$, so some bound on the degree is required to obtain sublinear multiplicity. 
In Section \ref{sec:highdeg}, we present a variant of Theorem \ref{thm:multbound1} (advertised in the abstract) which yields nontrivial bounds in the special case of simple $d-$regular graphs with degrees as large as $d=\exp(\log^{1/2-\delta }n)$, which is considerably larger than the regime $d=O(\polylog(n))$ handled by \cite{equiangular}.

\newcommand{\sone}{\frac{1}{4}\left(\frac{k}{\Delta^{7}\log \Delta}\right)^{1/5}}

\newcommand{\expone}{\exp\left(-\frac{k}{65\Delta^{7}s^4}\right)}
\newcommand{\exptwo}{\exp\left(-\frac{k}{100s^3}\right)}

\newcommand{\support}{\mathrm{support}}

The main new ingredient in the proof of Theorem \ref{thm:multbound1} is a polynomial lower bound on the support of (i.e., number of distinct vertices traversed by) a simple random walk of fixed length conditioned to return to its starting point. The bound holds for any connected graph and any starting vertex and may be of independent interest.


\begin{theorem}\label{thm:closed} Suppose $G$ is connected and of maximum degree $\Delta$ on $n$ vertices and $x$ is any vertex in $G$.  Let $\gamma_x^{2k}=(x=X_0,X_1,\ldots,X_{2k})$ denote a random walk of length $2k<n$ sampled according to the simple random walk on $G$ starting at $x$. Then
\begin{equation}
\P(\textnormal{support}(\gamma^{2k}_x)\leq s|X_{2k}=X_0)\leq \expone\qquad\textrm{for}\quad s\leq \sone.
\end{equation}
\end{theorem}
In particular, this means that for constant $\Delta$, the typical support of a closed random walk of length $2k$ is least $\Omega(k^{1/5})$.
It may be tempting to compare Theorem \ref{thm:closed} with the familiar fact that a random closed walk of length $2k$ on $\Z$ (or in continuous time, a standard Brownian bridge run for time $2k)$ attains a maximum distance of $\Omega(\sqrt{k})$ from its origin. However, as seen in Figure \ref{fig:biglolli}, there are regular graphs for which a closed walk of length $2k$ from a particular vertex $x$ travels a maximum distance of only $\polylog(k)$ with high probability. Theorem \ref{thm:closed} reveals that nonetheless the number of {\em distinct} vertices traversed is always typically $\poly(k)$. 
 We do not know if the specific exponent of $k^{1/5}$ supplied by Theorem \ref{thm:closed} is sharp, but considering a cycle graph shows that it is not possible to do better than $k^{1/2}$.
\begin{figure}\label{fig:biglolli}
    \centering
    \includegraphics[height=2in]{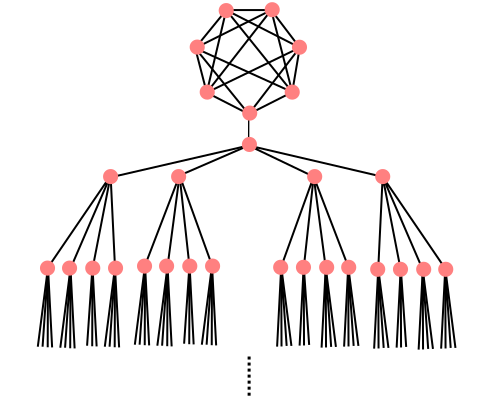}
    \caption{For a regular graph composed of a near-clique attached to an infinite tree, a closed walk of length $2k$ starting from within the near-clique does not typically go deeper than $O(\log k)$ down the tree. However, the support of such a closed walk is typically $k^{\Theta(1)}$. See Section \ref{sec:examples2} for a more detailed discussion.}

\end{figure}

 Given Theorem \ref{thm:closed}, our proof of Theorem \ref{thm:multbound1} follows the strategy of \cite{equiangular}: since most closed walks in $G$ have large support, the number of such walks may be drastically reduced by deleting a small number of vertices from $G$. By a moment calculation relating the spectrum to self return probabilities and a Cauchy interlacing argument, this implies an upper bound on the multiplicity of $\lambda_2(\tilde A_G)$. The crucial difference is that we are able to delete only $n/\polylog(n)$ vertices whereas they delete $n/\poly\log\log(n)$.

The key ingredient in our proof of Theorem \ref{thm:closed} is a result regarding the Perron eigenvector (i.e., the unique, strictly positive eigenvector with eigenvalue $\lambda_1$) of a submatrix of $\tilde A$. 
\begin{theorem}\label{thm:irregperron}
For any graph $G=(V,E)$ of maximum degree $\Delta$, take any set of vertices $S\subsetneq V$ such that the induced subgraph on $S$ is connected, and let $\psi_S$ be the  $\ell_2$-normalized Perron vector of $\tilde A_S$, the principal submatrix of $\tilde A$ corresponding to vertices in $S$. Then there is a vertex $u\in S$ which is adjacent to $V\setminus S$ such that
\begin{equation} \psi_S(u)\geq 1/(\Delta^{5/2}\lambda_1(\tilde A_S)|S|^{5/2}).\end{equation}
\end{theorem}

When we restrict this result to $G$ being a $d$-regular graph and pass to the adjacency matrix, we achieve a result about the unnormalized adjacency matrix of irregular graphs that may be of independent interest. 

\begin{corollary}\label{cor:irregperron1}
Let $H=(V,E)$ be an irregular connected graph of maximum degree $\Delta$ with at least two vertices, and let $\phi_H$ be the  $\ell_2$-normalized Perron vector of $A_H$. Then there is a vertex $u\in V$ with degree strictly less than $\Delta$ satisfying
\begin{equation} \phi_H(u)\geq 1/(\Delta^2\lambda_1(A_H)|V|^{5/2}).\end{equation}
\end{corollary}
Corollary \ref{cor:irregperron1} may be compared with existing results in spectral graph theory on the ``principal ratio'' between the largest and smallest entries of the Perron vector of a connected graph. 
The known worst case lower bounds on this ratio are necessarily exponential in the diameter of the graph \cite{CG,tait2015characterizing}, and it is known that the Perron entry of the highest degree vertex cannot be very small (see e.g. \cite[Ch. 2]{stevanovic2014spectral}). 
Corollary \ref{cor:irregperron1} articulates that there is always at least one vertex of non-maximal degree for which the ratio is only polynomial in the number of vertices. 

The proof of Theorem \ref{thm:irregperron} is based on an analysis of hitting times in the simple random walk on $G$ via electrical flows, and appears in Section \ref{sec:perron}. Combined with a perturbation-theoretic argument, it enables us to show that any small connected induced subgraph $S$ of $G$ can be extended to a slightly larger induced subgraph with significantly larger Perron value $\lambda_1(\tilde A_S)$. With some further combinatorial arguments, this implies that closed walks cannot concentrate on small sets, yielding Theorem \ref{thm:closed} in Section \ref{sec:closed}, which is finally used to deduce Theorem \ref{thm:multbound1} in Section \ref{sec:mult}.

We show in Section \ref{sec:examples3} via an explicit example (Figure \ref{fig:palm}) that the exponent of $5/2$ appearing in Corollary \ref{cor:irregperron1} is sharp up to polylogarithmic factors. We conclude with a discussion of open problems in Section \ref{sec:openprobs}.

\begin{figure}
    \centering
    \includegraphics[width=4.5in]{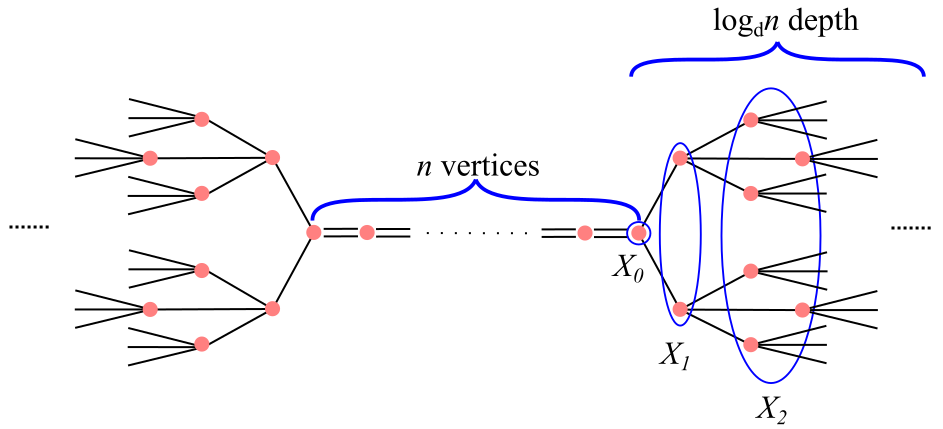}
    \caption{An example of a graph where all vertices $u$ that are not of maximum degree have $\psi(u)=\tilde O(n^{-5/2})$. The circled sets $X_0$, $X_1$ and $X_2$ will be used in the analysis of the graph in Section \ref{sec:examples3}.}
    \label{fig:palm}
\end{figure}

\begin{remark}[Higher Eigenvalues]
An update of the preprint of \cite{equiangular} generalizes Theorem \ref{thm:zhao} to the multiplicity of the $j$th eigenvalue. Our results can also be generalized in this manner by some nominal changes to the arguments in Section \ref{sec:mult}, but for simplicity we focus on $\lambda_2$ in this paper.
\end{remark}

\subsection{Higher degree regular graphs}\label{sec:highdeg}

If $G=(V,E)$ is a simple, $d$-regular graph, and $S\subsetneq V$ such that $|S|\leq d$, then necessarily all vertices of $S$ are adjacent to vertices in $V\setminus S$. Therefore we can improve the bound from Theorem \ref{thm:irregperron} by assuming the vertex on the boundary is the maximizer of the Perron vector, which has value $\psi_S(u)\geq 1/\sqrt{|S|}$. This leads to the following variants of our main results for simple, regular graphs of sufficiently high degree. 

\begin{theorem}\label{thm:highdegmult}
$G$ is simple, $d-$regular, and connected with $\lambda_2=\lambda_2(A_G)$, then
\begin{equation}\label{eqn:mainbound2}
    m_G\left([(1-\frac{\log\log_d n}{\log_d n})\lambda_2,\lambda_2]\right)=
    \begin{cases}
    \Ot\left(\frac{n}{d}\right) & \textrm{ when } d=o(\log^{1/4} n)\\
    \Ot\left(\frac{n\log^{1/2} d}{\log^{1/4} n}\right) &\textrm{ when }
    d=\Omega(\log^{1/4}n).
    \end{cases}
\end{equation}
\end{theorem}
The above theorem is based on the following  corresponding result for closed walks.
\begin{theorem}\label{thm:highdegwalk}
If $G$ is simple, $d-$regular, and connected on $n$ vertices and $\gamma$ is a random closed walk of length $2k<n$ started at any vertex in $G$, then:
\begin{equation}
\Pr(\support(\gamma)\le s)\leq \exptwo \qquad\textrm{for}\quad s\le \min\left\{\frac{1}{8}\left(\frac{k}{\log d}\right)^{1/4}, \frac{d}{2}\right\}.
\end{equation}
\end{theorem}

The proofs of both theorems appear in Appendix \ref{sec:highdegproofs}

\subsection{Related work}\label{sec:related}

\paragraph{Eigenvalue Multiplicity.} 
Despite the straightforward nature of the question, relatively little is known about eigenvalue multiplicity of general graphs. As discussed in \cite{equiangular}, if one assumes that $G$ is a bounded degree expander graph, then the bound of Theorem \ref{thm:zhao} can be improved to $O(n/\log n)$. On the other hand, if $G$ is assumed to be a Cayley graph of bounded doubling constant $K$ (indicating non-expansion), then \cite{lee2008eigenvalue} show that the multiplicity of the second eigenvalue is at most $\exp(\log^2 K)$. In the context of Cayley graphs, one interesting new implication of Theorem \ref{thm:highdegmult} is that all Cayley graphs of degree $O(\exp(\log^{1/2-\delta}n))$ have second eigenvalue multiplicity $O(n/\log^{\delta/2}n)$.

Distance regular graphs of diameter $D$ have exactly $D+1$ distinct eigenvalues (see \cite{Godsil} 11.4.1 for a proof). However, besides the top eigenvalue (which must have multiplicity 1), generic bounds on the multiplicity of the other eigenvalues are not known. As expanding graphs have diameter $\Theta(\log_{d}n)$, the average multiplicity of eigenvalues besides $\lambda_1$ for expanding distance regular graphs is $\Theta(n/\log_d n)$. It is tempting to see this as a hint that the multiplicity of the second eigenvalue could be $\Omega(n/\log_d n)$.

Sublinear multiplicity does not necessarily hold for eigenvalues in the interior of the spectrum even assuming bounded degree.  In particular, Rowlinson has constructed connected $d-$regular graphs with  an eigenvalue of multiplicity at least $n(d-2)/(d+2)$ \cite{Row} for constant $d$.

\paragraph{Higher Order Cheeger Inequalities.} The results of \cite{louis2012many, lee2014multiway} imply that if a $d-$regular graph $G$ has a second eigenvalue multiplicity of $m$, then its vertices can be partitioned into $\Omega(m)$ disjoint sets each having edge expansion $O(\sqrt{d(1-\lambda_2)\log m})$. Combining this with the observation that a  set cannot have expansion less than the reciprocal of its size shows that $m=O(n/\polylog(n))$ whenever  $1-\lambda_2(\tilde A_G) \le 1/\log^{c}n$ for any $c>1$, i.e., the graph is sufficiently non-expanding. Our main theorem may be interpreted as saying that this phenomenon persists for all graphs. 

\paragraph{Support of Walks.} 
There are as far as we are aware no known lower bounds for the support of a random closed walk of fixed length in a general graph (or even Cayley graph). It is relatively easy to derive such bounds for bounded degree graphs if the length of the walk is sufficiently larger than the mixing time of the simple random walk on the graph; the key feature of Theorem \ref{thm:closed}, which is needed for our application, is that the length of the walk can be taken to be much smaller.

The support of open walks (namely removing the condition that the walk ends at the starting point) is better understood. There are Chernoff-type bounds on the size of the support of a random walk based on the spectral gap \cite{gillman1998chernoff,kahale1997large}. Such bounds and their variants are
an important tool in derandomization.

\paragraph{Entries of the Perron Vector.} There is a large literature concerning the magnitude of the entries of the Perron eigenvector of a graph --- see \cite[Chapter 2]{stevanovic2014spectral} for a detailed discussion of results up to 2014. Rowlinson showed sufficient conditions on the Perron eigenvector for which changing the neighborhood of a vertex increases the spectral radius \cite{rowlinson1990more}. Cvetkovi{\'c}, Rowlinson, and Simi{\'c} give a condition which, if satisfied, means a given edge swap increases the spectral radius \cite{cvetkovic1993study}. Cioab\u{a} showed that for a graph of maximum degree $\Delta$ and diameter $D$,  $\Delta-\lambda_1>1/nD$ \cite{cioabua2007spectral}. Cioab\u{a}, van Dam, Koolen, and Lee then showed that $\lambda_1\ge (n-1)^{1/D}$ \cite{cioabua2010lower}. The results of \cite{van2011decreasing} prove a lemma similar to Lemma \ref{lem:test}, giving upper and lower bounds on the change in spectral radius from the deletion of edges. However, their result does not quite imply Lemma \ref{lem:test}, and we prove a slightly different statement.

\subsection{Notation}
All logarithms are base $e$ unless noted otherwise.

 \paragraph{Electrical Flows.}
 We use $\reff_{H}(\cdot,\cdot)$ to denote the effective resistance between two vertices in $H$, viewing each edge of the graph as a unit resistor. See e.g. \cite{doyle1984random} or \cite[Chapter IX]{MGT} for an introduction to electrical flows and random walks on graphs.

\paragraph{Graphs.}
For a matrix $M$, we use $M_S$ to denote the principal submatrix of $M$ corresponding to the indices in $S$. Consider a graph $G=(V,E)$ and a subset $H\subset V$. Let $P:=AD^{-1}$ be the transition matrix of the simple random walk matrix on $G$, where $A$ is the adjacency matrix and $D$ is the diagonal matrix of degrees. We will  also use the normalized adjacency matrix $\tilde A:= D^{-1/2}AD^{-1/2}$. Note that $P$ and $\tilde A$ are similar, and that $\tilde A$ is symmetric. $P_S$ and $\tilde A_S$ are submatrices of $P$ and $\tilde A$; they are not the  transition matrices and normalized adjacency matrices of the induced subgraph on $S$. Note $P_S$ and $\tilde A_S$ are also similar.

\paragraph{Perron Eigenvector.}
We use $\psi_S$ to denote the $\ell_2$-normalized eigenvector corresponding to
$\lambda_1(\tilde A_S)$, which is a simple eigenvalue if $S$ is connected. Note that for connected $S$, $\psi_S$ is strictly positive
 by the Perron-Frobenius theorem.

A simple graph refers to a graph without multiedges or self-loops. We assume $\Delta\ge 2$ for all connected regular graphs, since otherwise the graph is just an edge, so $\log \Delta>0$.

\section{Lower Bounds on the Perron Eigenvector}\label{sec:perron}
In this section we prove Theorem \ref{thm:irregperron}, which is a direct consequence of the following slightly more refined result. In a graph $G=(V,E)$, define the boundary of $S$ as the set of vertices in $S$ adjacent to $V\backslash S$ in $G$.
\begin{theorem}[Large Perron Entry]\label{lem:electric}
Let $G=(V,E)$ be a connected graph of maximum degree $\Delta$ and $S\subsetneq V$ such that the induced subgraph on $S$ is connected. Then there is a vertex $u\in S$ on the boundary of $S$ such that
\begin{equation}\label{eqn:psilarge}\psi_S(u)/\psi_S(t)\geq 1/(\Delta^{5/2}\lambda_1(\tilde A_S)|S|^2)\end{equation}
where $ t=\arg \max_{w\in S} \psi_S(w)$.
\end{theorem}

At a high level, the proof proceeds as follows.
First, we show that there exists a vertex $x\in S$ adjacent to the boundary of $S$ such that a random walk started at $x$ is somewhat likely to hit $t$ before it hits the boundary of $S$. Second, we express the ratio of $D^{1/2}_S\psi_S(x)$ and $D^{1/2}_S\psi_S(t)$ as a limit as $k\to \infty$ of the ratio $\P{Y^k_x}/\P{Y_t^k}$, where $Y^k_v$ is the event that the simple random walk started at $v$ remains in $S$ for $k$ steps; we bound this ratio from below using the hitting time estimate from the first step. 
Third, by the eigenvector equation the ratio of the entries of an eigenvector at two neighboring vertices is bounded. Hence, $x$ is adjacent to some vertex $u$ on the boundary of $S$ satisfying the theorem.

\begin{proof}
    Write $S=M\sqcup B$, where $B$ is the boundary of $S$ and $M=S\setminus B$. If $t\in B$ then we are done, so assume not. 
	Let $\P^G_x(\cdot)$ denote
	the law of the simple random walk (SRW) $(X_i)_{i=0}^\infty$ on $G$ started at $X_0=x$, and for any subset $T\subset V$, let
	$\T_T :=\{\min i: X_i\in T\}$
	denote the hitting time of the SRW to that subset; if $T=\{u\}$ is a singleton we will simply write $\T_u$. \\
	
    \noindent {\em Step 1.} We begin by showing that there is a vertex $x\in M$ adjacent to $B$ for
	which the random walk started at $x$ is reasonably likely to hit $t$ before $B$.
	To do so, we use the well-known connection between hitting probabilities in random walks
	and electrical flows.
	Define a new graph $K=(V'=V\setminus B\cup\{s\},E')$ by contracting all vertices in $B$ to a single
	vertex $s$. Let $f:V'\rightarrow [0,1]$ be the vector of voltages in the 
	electrical flow in $K$ with boundary conditions $f(s)=0, f(t)=1$, regarding every edge as a unit resistor. 
	By Ohm's law, the current flow from $s$ to $t$ is equal to $1/\reff_{K}(s,t)$. We have the crude upper bound $$\reff_{K}(s,t)\le\mathrm{distance}_K(s,t)\le |S|,$$
	since $S$ is connected,
	so the outflow of current from $s$ is at least $1/|S|$. By Kirchhoff's current law, there must be a flow 
	of at least $1/(|S|\deg_{K}(s))$ on at least one edge $(s,x)\in E'$. By Ohm's law again, for this particular $x\in V'$ we must have
	\begin{equation}\label{eqn:ohm} f(x)\ge \frac{1}{|S|\deg_{K}(s)} = \frac{1}{|S||\partial_G B|}\ge \frac{1}{\Delta |S|^2},\end{equation}
	where $\partial_G B$ denotes the edge boundary of $B$ in $G$.
	Appealing to e.g. \cite[Chapter IX, Theorem 8]{MGT}, this translates to the probabilistic bound
	\begin{equation}\label{eqn:hittingst} \P_x^G(\T_t<\T_B)=\P_x^{K}(\T_t<\T_s)  = f(x)\ge  \frac{1}{\Delta |S|^2}.\end{equation}
	
	Finally, since $f(s)=f(y)=0$ for every $y\in V\setminus S$, we must in fact have $x\in M$.\\

\begin{figure}
    \centering
    \includegraphics[height=2in]{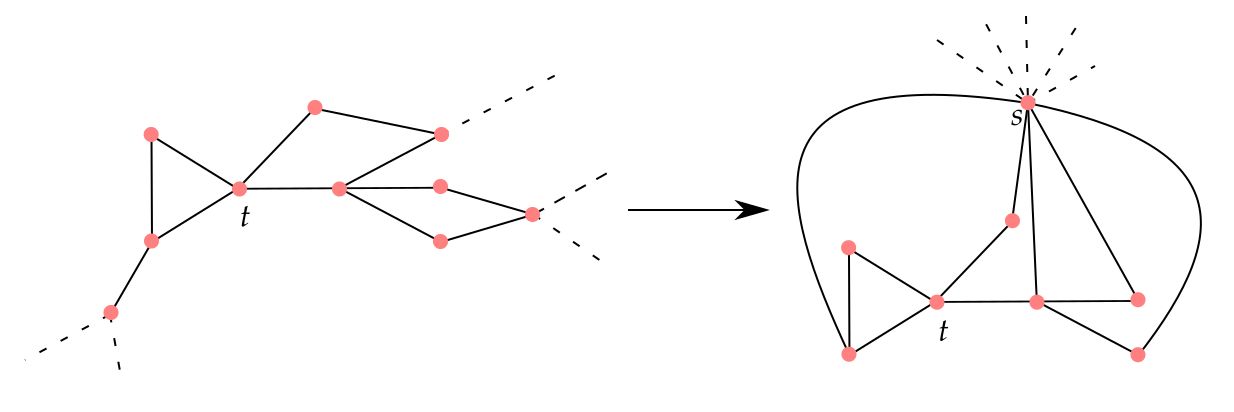}
    \caption{ In Step 1 of the proof of Theorem \ref{lem:electric}, we lower bound the probability that a random walk started at a certain vertex $x$ adjacent to $B$ reaches $t$ before reaching $B$. We do this by contracting $B$ to a vertex $s$, then lower bounding the current from $s$ to $t$, which establishes the existence of the desired $x$. The left graph in the figure is $G$ and the right graph is the contracted graph $K$, with the dotted lines indicating edges leaving the set of interest $S=M\sqcup B$. } 
    \label{fig:contractst}

\end{figure}
		\noindent {\em Step 2.} We now use \eqref{eqn:hittingst} to show that $\psi_S(x)$ is large. Because $\tilde A_S = D^{-1/2}_SP_SD^{1/2}_S$, the top eigenvector of $P_S$ is $D^{1/2}_S\psi_S/\|D^{1/2}_S\psi_S\|$. Let $P':(P+I)/2$ denote the lazy random walk\footnote{This modification is only to ensure non-bipartiteness; if $S$ is not bipartite we may take the simple random walk} on $G$, and to ease notation let $\P'_x(\cdot):={\P'_x}^{G}(\cdot)$ denote the law of the lazy random walk on $G$ started at $x$.
	Note that the eigenvectors of $P_S$, as well as $\P_x(\T_t<\T_B)$, do not change when passing to $P'_S$.

For the lazy random walk, the Perron-Frobenius theorem implies that
	$$\frac{ (D^{1/2}_S\psi_{S})(w)}{\|D^{1/2}_S\psi_{S}\|}= \lim_{k\rightarrow \infty} \frac{\one_S^T {P'_{S}}^k e_w}{\|\one_S^T {P'_{S}}^k\|},$$
    for every $w\in S$, where $\one_S\in\R^{S}$ is the all ones vector.
We further have
	$$ \one_S^T {P'_{S}}^k e_w = \P'_w(\tau_{V\backslash S}>k),$$ namely the probability a random walk of length $k$ starting at $w$ stays in $S$.
	
	We are interested in the ratio
	\begin{equation}\label{eqn:limratio}
	 \frac{(D^{1/2}_S\psi_S)(x)}{(D^{1/2}_S\psi_S)(t)}=\lim_{k\rightarrow\infty}\frac{\P'_x(\tau_{V\backslash S}>k)}{\P'_t(\tau_{V\backslash S}>k)}.
	\end{equation}
	Fix an integer $k>0$. The numerator of \eqref{eqn:limratio} is bounded as
	\begin{align}
	\P'_x(\tau_{V\backslash S}>k)&\geq \P'_x(\tau_{V\backslash S}>k|\T_t<\T_B)\P'_x (\T_t<\T_B)\nonumber\\
						&\ge \frac{1}{\Delta |S|^2}\P'_x(\tau_{V\backslash S}>k|\T_t<\T_B)\qquad\textrm{by \eqref{eqn:hittingst}}\nonumber \\
						&\ge \frac{1}{\Delta |S|^2}\sum_{\theta=0}^{k-1}\P'_x(\tau_{V\backslash S}>k|\T_t=\theta,\T_t<\T_B)\P'_x(\T_t=\theta|\T_t<\T_B)\\
						&=\frac{1}{\Delta |S|^2} \sum_{\theta=0}^{k-1}\P'_t(\tau_{V\backslash S}>{k}-\theta)\P'_x(\T_t=\theta|\T_t<\T_B)\nonumber\\
						&\ge\frac{1}{\Delta |S|^2} \sum_{\theta=0}^{k-1}\P'_t(\tau_{V\backslash S}>{k})\P'_x(\T_t=\theta|\T_t<\T_B). \label{eqn:thetasum}
	\end{align}
    Observe that $\E'_x \T_B<\infty$
    since $G$ is connected.
	Thus, 
		\begin{align*}
			\sum_{\theta=0}^{k-1} \P'_x (\T_t=\theta | \T_t<\T_B) &= 1-\P'_x(\T_t\ge k|\T_t<\T_B) \nonumber\\
			&\ge 1- \frac{\P'_x(\T_B\ge k)}{\P'_x(\T_t<\T_B)}\nonumber\\
			&\ge 1-\frac{\E'_x \T_B}{k}\cdot \Delta |S|^2\quad\textrm{by Markov and \eqref{eqn:hittingst}}.
		\end{align*}
	Combining this bound with \eqref{eqn:thetasum}, we have
	$$ \P'_x(\tau_{V\backslash S}>k) \ge \frac{1}{\Delta |S|^2}\left(1-\frac{\E'_x \T_B}{k}\cdot \Delta |S|^2\right)\P'_t(\tau_{V\backslash S}>k)$$

	Taking the limit as $k\rightarrow\infty$ in \eqref{eqn:limratio} yields
	$$\frac{(D^{1/2}_S\psi_S)(x)}{(D^{1/2}_S\psi_S)(t)}\ge \frac{1}{\Delta |S|^2}.$$\\

\noindent {\em Step 3.}	Since $x$ is adjacent to $B$, we can choose a $u\in B$ adjacent to $x$. The eigenvector equation and nonnegativity of the Perron vector now imply
	$\Delta\lambda_1(A_S)\psi_S(u) \ge \psi_S(x)$, whence 
	\begin{equation}\label{eqn:transitionperron}
	(D^{1/2}_S\psi_S)(u)\ge \frac{1}{\lambda_1(\tilde A_S)\Delta^2 |S|^2}(D^{1/2}_S\psi_S)(t).
	\end{equation}
	
	Therefore, as $D$ is a diagonal matrix, and the entries of $D$ range from $1$ to $\Delta$, it must be the case that
	\[
	\psi_S(u)\ge \frac{1}{\lambda_1(\tilde A_S)\Delta^{5/2} |S|^2}\psi_S(t).
	\]
\end{proof}
\begin{remark} As the proof shows, the right-hand side of \eqref{eqn:psilarge} may be replaced with $1/\Delta^{3/2}\lambda_1(\tilde A_S)|\partial_G B| R$
	where $B$ is the boundary of $S$ in $G$ and $R$ is the maximum effective resistance between two vertices in $S$.
\end{remark}
\begin{proof}[Proof of Corollary \ref{cor:irregperron1}]
    Given an irregular graph $H$, construct a $\Delta-$regular graph $G$ containing $H$ as an induced subgraph (it is trivial to do this if we allow $G$ to be a multigraph). Repeating the above proof on $G$ with $S=H$ and observing that $D_S^{1/2}$ is a multiple of the identity since $G$ is regular, \eqref{eqn:transitionperron} yields the desired conclusion.
\end{proof}

\section{Support of Closed Walks}\label{sec:closed} \newcommand{\shrinkage}{\left(1+\frac{5}{128\Delta^7s^4}\right)}

\newcommand{\stwo}{\min\left\{\frac{1}{8}\left(\frac{k}{\log d}\right)^{1/4}, \frac{d-h}{2}\right\}}

In this section we prove Theorem \ref{thm:closed}, which is an immediate consequence of the following slightly stronger result. Let $W^{2k, s}$ denote the event a simple random walk of length $2k$ has support at most $s$ and ends at its starting point. 

\begin{thm}[Implies Theorem \ref{thm:closed}]\label{thm:cyclesup} If $G$ is connected and of maximum degree $\Delta$ on $n$ vertices, then for every vertex $x\in G$ and $k<n/2$,
\begin{equation}\label{eqn:lowcycles}
\P_x(W^{2k,s})\leq \expone\P_x(W^{2k,2s})\qquad\textrm{for}\quad s\leq \sone.
\end{equation}

\end{thm}

The proof requires a simple lemma lower bounding the increase in the Perron value of a subgraph upon adding a vertex in terms of the Perron vector.
\begin{lemma}[Perturbation of $\lambda_1$] \label{lem:test} Take the normalized adjacency matrix $\tilde A:=D^{-1/2}AD^{-1/2}$ of a graph $G=(V,E)$ of maximum degree $\Delta$. For any $S\subsetneq V$ and vertex $u\in S$, the
submatrix which includes the subset $S'=(S\cup \{v\},E(S)\cup \{(u,v)\})$, which adds a vertex $v$ and
	the edge $(u, v)$ to $S$, satisfies
	\[\lambda_1(\tilde A_{S'})\geq\frac{1}{2}\left(\lambda_1(\tilde A_S)+\sqrt{\lambda_1(\tilde A_S)^2+\Delta^{-2}\psi_S(u)^2}\right).\]
\end{lemma}

\begin{proof}
The largest eigenvalue  of $\tilde A$ is at least the quadratic form {associated with} the unit vectors \[
g_\alpha(x)=\left\{\begin{array}{cc}
     \sqrt{1-\alpha^2}\psi_S(x)&x\in V  \\
     \alpha& x=v
\end{array}\right.
\]
for $0\leq \alpha\leq 1$. We have
$g_\alpha^T\tilde Ag_\alpha=(1-\alpha^2)\lambda_1(\tilde A _S)+d_{u}^{-1/2}d_v^{-1/2}\alpha\sqrt{1-\alpha^2}\psi_S(u)$, where $d_u$ is the degree of $u$ in $G$. This quantity is maximized when
\[
\alpha=\sqrt{\frac12-\frac{\lambda_1(\tilde A_S)}{2\sqrt{\lambda_1(\tilde A_S)^2+d_u^{-1}d_v^{-1}\psi_S(u)^2}}},
\]
at which
\[
g_\alpha^T\tilde Ag_\alpha=\frac{1}{2}\left(\lambda_1(\tilde A_S)+\sqrt{\lambda_1(\tilde A_S)^2+d_u^{-1}d_v^{-1}\psi_S(u)^2}\right).
\]
\end{proof}

Combining Lemma \ref{lem:test} and Theorem \ref{lem:electric} yields a bound on the increase of the top eigenvalue of the submatrix corresponding to an induced subgraph that may be achieved by adding vertices to it.
\begin{lemma}[Support Extension] \label{lem:increase}
For any connected graph $G=(V,E)$ of maximum degree $\Delta$, consider its normalized adjacency matrix $\tilde A$. For any connected subset $S\subsetneq V$ such that $2\le |S|=s<|V|/2$, there is a connected subset $T\subset V$ containing $S$ such that $|T|=2s$ and
\[
\lambda_1(\tilde A_{T})\ge \lambda_1(\tilde A_S)\left(1+\frac{5}{128\Delta^7s^4}\right).
\]
\end{lemma}

\begin{proof}
	Define $\lambda_1:=\lambda_1(\tilde A_S)$ and note that $\lambda_1\geq 1/\Delta$ since $S$ contains at least one edge. As $\psi_S$ is a normalized vector with $s$ entries, $\psi_S(t)\geq1/\sqrt{s}$. Therefore $\psi_S(u)\geq 1/(\Delta^{5/2}\lambda_1 s^{5/2})$. Take $v$ to be any vertex in $V\setminus S$ that neighbors $u$ in $G$. By Lemma \ref{lem:test},

\begin{align}
\lambda_1(\tilde A_{S\cup\{v\}})&\geq\frac{1}{2}\left(\lambda_1+\sqrt{\lambda_1^2+\Delta^{-2}\psi_S(u)^2}\right)\nonumber\\
&\geq \lambda_1+\frac{\psi_S(u)^2}{4\lambda_1\Delta^2}-\frac{\psi_S(u)^4}{16\lambda_1^3\Delta^4}\nonumber\\
&\geq \lambda_1+\frac1{6\lambda_1^3\Delta^7s^5}\qquad\textrm{as $\psi_S(u)^2/\lambda_1^2 \le \Delta^2$} \nonumber\\
&\geq \lambda_1+\frac1{6\Delta^7s^5}\qquad\textrm{since $\lambda_1\le 1$.}\label{eq:onevert}
\end{align}

Assuming that $s<|V|/2$, we can iterate this process $s$ times, adding the vertices $\{v_1,\ldots v_s\}$. At each step we add the vertex $v_i$ and increase the Perron eigenvalue of $\tilde A_{S\cup\{v_1,\ldots,v_{i-1}\}}$ by at least $1/(6\Delta^7(s+i-1)^5)$. Therefore, defining $T=S\cup\{v_1,\ldots v_s\}$, we have
\[
\lambda_1(\tilde A_T)\geq \lambda_1+\frac1{6\Delta^7}\sum_{i=1}^s \frac{1}{(s+i-1)^5}\geq \lambda_1 +\frac{5}{128\Delta^7s^4},
\]
where the last inequality follows from approximating the sum with the integral. As $\lambda_1\leq 1$, this translates to the desired multiplicative bound.
\end{proof}

\begin{proof}[Proof of Theorem \ref{thm:cyclesup}] We begin by showing \eqref{eqn:lowcycles}.
Let $\Gamma_x^{s}$ be the set of connected subgraphs of $G$ with $s$ vertices containing $x$. Choose $S$ to be the maximizer of $e_x^T\tilde A_S^{2k} e_x$ among $S\in \Gamma_x^{s}$, and let $T\in \Gamma_x^{2s}$ be the extension of $S$ guaranteed by Lemma \ref{lem:increase} to satisfy
$$ \lambda_1(\tilde A_T)\ge \shrinkage \lambda_1(\tilde A_S).$$
$P^{2k}_S$ has the same diagonal entries as $\tilde A^{2k}_S$, so
$$\P_x(W^{2k,s}) \le \sum_{S'\in \Gamma_x^{s}} e_x^T\tilde A_{S'}^{2k}e_x,$$
since each walk of length $2k$ satisfying $W^{2k,s}$ is contained in at least one $S'\in \Gamma_x^{s}$. Furthermore, $|\Gamma_x^{s}|\le \Delta^{2s}$ since each subgraph of $\Gamma_x^s$ may be encoded by one of its spanning trees, which may in turn be encoded by a closed walk rooted at $x$ traversing the edges of the tree. We then have
\begin{align}
    \P_x(W^{2k,s}) &\le |\Gamma_x^s| e_x^T\tilde A_{S}^{2k} e_x\nonumber\\
                & \le \Delta^{2s} \lambda_1(\tilde A_S)^{2k}\nonumber\\
                & \le \Delta^{2s} \shrinkage^{-2k} \lambda_1(\tilde A_T)^{2k}.\label{eqn:w2ks}
\end{align}
We will  bound the right hand side in terms of $\P_x(W^{2k,2s})$.

We claim that for every $z\in T$, 
\begin{equation}\label{eqn:xzwalks} e_x^T\tilde A_T^{2k} e_x \ge \Delta^{-4s}e_z^T\tilde A_T^{2k-4s}e_z.\end{equation}
To see this, let $\pi$ be a path in $T$ of length $\ell\le 2s$ between $x$ and $z$, which must exist since $T$ is connected and has size $2s$. Then every closed walk of length $2k-2\ell$ in $T$ rooted at $z$ may be extended to a walk of length $2k$ in $T$ rooted at $x$ by attaching $\pi$ and its reverse. Performing the walk of $\pi$ twice occurs with probability at least $\Delta^{-2\ell}$. Since all of the walks produced this way are distinct, we have
$$e_x^T \tilde A_T^{2k}e_x\ge \Delta^{-2\ell} e_z^T \tilde A_T^{2k-2\ell}e_z.$$
By the same argument $e_z^T\tilde A_T^{2k-2\ell}e_z\ge \Delta^{-4s+2\ell} e_z^T\tilde A_T^{2k-4s}e_z$,  and inequality \eqref{eqn:xzwalks} follows.

Choose $z\in T$ to be the maximizer of $e_z^T \tilde A_{T}^{2k-4s}e_z$, for which we have:
$$ e_z^T\tilde A_T^{2k-4s} e_z \ge \frac{1}{2s}\tr(P_T^{2k-4s}) \ge \frac{\lambda_1(\tilde A_T)^{2k-4s}}{2s}.$$
Combining this with \eqref{eqn:xzwalks} and substituting in \eqref{eqn:w2ks}, we obtain
\begin{align*} 
\P_x(W^{2k,s}) &\le \Delta^{6s}\cdot 2s \shrinkage^{-2k}\lambda_1(\tilde A_T)^{4s} e_x^T\tilde A_T^{2k} e_x\\ &\le \Delta^{6s}\cdot 2s \shrinkage^{-2k}\lambda_1(\tilde A_T)^{4s} \P_x(W^{2k,2s}).
\end{align*}

Applying the inequality $e^{x/2}\leq 1+x$ for $0<x<1$ and the bound $\lambda_1(\tilde A_T)<1$, we obtain
\begin{equation}\label{eqn:wbound}
\P_x(W^{2k,s})\leq \exp\left(6s\log \Delta+\log (2s)-\frac{5k}{128\Delta^{7}s^4}\right)\P_x(W^{2k,2s}),
\end{equation}
which implies
$$
\P_x(W^{2k,s})\leq \exp\left(-\frac{k}{65\Delta^{7}s^4}\right)\P_x(W^{2k,2s})
$$
whenever $$s\leq \frac 14 \left(\frac{k}{\Delta^{7}\log (\Delta)}\right)^{1/5},$$ establishing \eqref{eqn:lowcycles}.

\end{proof}

\section{Bound on Eigenvalue Multiplicity}\label{sec:mult}
In this section we prove Theorem \ref{thm:multbound1}, restated below in slightly more detail.
\begin{theorem}[Detailed Theorem \ref{thm:multbound1}]
Let $G$ be a maximum degree $\Delta$ connected graph on $n$ vertices. If\footnote{If $\Delta\ge \log^{1/7}n/\log\log n$ then \eqref{eqn:mainbound1} is vacuously true.} $\Delta\le \log^{1/7}n/\log\log n$ then the spectrum of the normalized adjacency matrix $\tilde A$ satisfies
\begin{equation}
  m_G\left([(1-\frac{\log\log_\Delta n}{\log_\Delta n})\lambda_2,\lambda_2]\right)= O\left(n\cdot\frac{\Delta^{7/5}(\log^{2/5}\Delta)\log\log n}{\log^{1/5}n}\right).
\end{equation}

\end{theorem}

\begin{proof}

 For now, assume that $\abs{\lambda_n(P)}\leq \abs{\lambda_2(P)}$.  Let $\P(\cdot)$ denote the law of an SRW $\gamma$ of length $2k$ on $G$, started at a vertex chosen uniformly at random (i.e., {\em not} from the stationary measure of the SRW). Let $W^{2k}:=W^{2k,n}$ denote the event that $\gamma$ returns to its starting vertex after $2k$ steps. In an abuse of notation, let $W^{2k,\geq s+1}:=W^{2k}\backslash W^{2k,s}$ be the event that a walk of length $2k$ is closed and has support at least $s+1$.

Set $k:=\frac13\log_\Delta n$ and $c:=2\log k$ and let $s$ be a parameter satisfying 
\begin{equation}\label{eqn:sets} \P(W^{2k,s})\le e^{-c}\P(W^{2k})\end{equation}
to be chosen later. Delete $cn/s$  vertices from $G$ uniformly at random and call the resulting graph $H$. 

If $\gamma$ has support at least $s+1$, then the probability that none of the vertices of $\gamma$ are deleted is at most
\[\left(1-\frac sn\right)^{\frac{cn}{s}}\leq e^{-c}.\]
Thus, 
$$\E_H \P(\gamma\subset H|\gamma\in W^{2k,\ge s+1})\le e^{-c},$$
where $\E_H$ is the expectation over $H$.
 It then follows by the probabilistic method that there exists a deletion such that the resulting subgraph $H$ of $G$ satisfies $$\P(W^{2k,\ge s+1}\cap \{\gamma\subset H\})\le e^{-c}\P(W^{2k, \geq s+1}).$$

 Write $\lambda_2:=\lambda_2(\tilde A_G)$ and let $m'$ be the number of eigenvalues of $H$ in the interval $[(1-\epsilon)\lambda_2,\lambda_2]$ for $\epsilon:=c/2\log_\Delta(n)$. Since $2k$ is even,
\begin{align*}
m'(1-\epsilon)^{2k}\lambda_2^{2k} &\leq \tr(\tilde A_H^{2k})\\
&=n\P(W^{2k}\cap\{\gamma\subset H\})\\
&=n(\P(W^{2k,s}\cap\{\gamma\subset H\})+\P(W^{2k,\ge s+1}\cap \{\gamma\subset H\}))\\
&\le n(\P(W^{2k,s})+e^{-c}\P(W^{2k,\geq s+1}))\quad\textrm{by our choice of }H\\
&\le n(e^{-c}\P(W^{2k})+e^{-c}\P(W^{2k,\geq s+1}))\quad\textrm{by \eqref{eqn:sets}}\\
&\leq 2e^{-c}\tr(\tilde A_G^{2k})\\
&\leq  2e^{- c}(n\lambda_2^{2k}+1).
\end{align*}
We may assume that the diameter of $G$ is at least $4$ as otherwise $\Delta\geq n^{1/4}$, making the theorem statement vacuous. Because of the diameter, we can take two edges $(u,v)$ and $(a,b)$, such that the distance between the edges is at least $2$. Then consider the vector $\phi$ such that
\[
\phi(x)=\left\{
\begin{array}{cc}
    1 &  x\in \{u,v\}\\
     -1& x\in\{a,b\}\\
     0&\textnormal{ otherwise}
\end{array}
\right.
\]

As the top eigenvector of the normalized adjacency matrix is the all ones vector, $\phi$ is orthogonal to it. Therefore by Courant Fisher
\[
\lambda_2\geq \frac{\phi^TD^{-1/2}AD^{-1/2}\phi}{\phi^T\phi}\geq \frac{4\Delta^{-1}}{4}= \frac{1}{\Delta}.
\]
By our choice of $k$, this means $n\lambda_2^{2k}\geq 1$. Moreover,

\begin{align*}
    \epsilon \le \frac{2\log\log n}{2\log_\Delta n}\le \frac{\log \Delta\log\log n}{\log n}<1/2,
\end{align*}
based on our assumptions on $\Delta$. Thus, $1-\epsilon\geq e^{-1.5\epsilon }$. Combining these facts,
\[
m'\lambda_2^{2k}\leq 4e^{3k\epsilon- c}n\lambda_2^{2k},
\]
yielding
\[
m'\leq 4ne^{3k\epsilon-c}\le 4ne^{-c/2}=O\left(\frac{n}{\log_\Delta n}\right).
\]
As we created $H$ by deleting $cn/s$ vertices, it follows by Cauchy interlacing that the number of eigenvalues of $\tilde A$ in $[(1-\epsilon)\lambda_2,\lambda_2]$ is at most
\[
\frac {cn}{s}+O\left(\frac{n}{\log_\Delta n}\right).
\]

We now show that taking
$$ s:=\frac{1}{4}\left(\frac{k}{\Delta^{7}\log \Delta}\right)^{1/5}$$
satisfies \eqref{eqn:sets}.
Applying Theorem \ref{thm:cyclesup} equation \eqref{eqn:lowcycles} to each $x\in G$ and summing, we have
\begin{align*}
    \frac{\P(W^{2k,s})}{\P(W^{2k})} &\le \exp\left(-\frac{k}{65\Delta^{7}s^4}\right)\\
    &\le \exp \left(-\Omega\left(\frac{\log n\log^{2/5}\Delta}{\Delta^{7/5}}\right)\right)\\
    &\ll \exp(-c)=\exp(-\Theta(\log\log_\Delta n)),
\end{align*}
		satisfying \eqref{eqn:sets} for sufficiently large $n$,
and we conclude that
$$m_G\left([(1-\frac{\log\log_\Delta n}{\log_\Delta n})\lambda_2,\lambda_2]\right)=O\left(n\cdot\frac{\Delta^{7/5}\log^{2/5}\Delta\log\log n}{\log^{1/5}n}\right),$$
as desired.

If $\abs{\lambda_n}> \abs{\lambda_2}$, we can do a lazy walk with probability of moving $p=\frac12$, therefore making all eigenvalues nonnegative. This is equivalent to doubling the degree of every vertex by adding loops. This is the equivalent of taking the simple random walk on a graph with maximum degree $2\Delta$, requiring $s\leq \frac 1{11} \left(\frac{k }{\Delta^{7}\log \Delta}\right)^{1/5}$, yielding the same asymptotics. 
\end{proof}

\section{Examples}\label{sec:examples}
In this section, we consider examples demonstrating some of the points raised in the introduction regarding the tightness of our results. As most of our results in this section are combinatorial rather than probabilistic, we will consider multiplicity in the non-normalized adjacency matrix $A$. For regular graphs, this is equivalent.
\subsection{Bipartite Ramanujan Graphs}\label{sec:examples1}
 We show that bipartite Ramanujan graphs (see \cite{lubotzky1988ramanujan}; known to exist for every $d\ge 3$ by \cite{marcus2015interlacing}) have high multiplicity near $\lambda_2$. This means that the bound of $n/\log^{\Theta(1)}n$ of Theorem \ref{thm:multbound1} is tight. 

\begin{lemma}[McKay \cite{McKay81} Lemma 3]\label{lem:mckay}
	The number of closed walks on the infinite $d$-regular tree of length $2k$ starting at a fixed vertex is $\Theta\left( \frac{4^k(d-1)^k}{k^{3/2}} \right)$.
\end{lemma}
\begin{prop}
	There exists a constant $\alpha>0$ such that for fixed $d$, every bipartite $d$-regular bipartite Ramanujan graph $G$ on $n$ vertices satisfies

	$$m_G\left([\lambda_2(1-\alpha\frac{\log\log(n)}{\log(n)}), \lambda_2]\right)=\Omega(n/\log^{3/2}(n)).$$
\end{prop}
\begin{proof}
	Let $k=\beta\log n$ for some constant $\beta$ to be set later and suppose that there are $m$ eigenvalues of $A_G$ in the interval $[\lambda_2\left(1-\alpha\frac{\log\log(n)}{\log(n)}\right),\lambda_2]$. Recall that the spectrum of a bipartite graph is symmetric around 0. From \cref{lem:mckay} it follows that for some constant $C$,
	\begin{align*}
		Cn\left( \frac{4^k(d-1)^k}{k^{3/2}} \right)
		&\leq \sum_{i=1}^n\lambda_i(A_G)^{2k} 
		\\ &\leq 2d^{2k}+(n-2m)\left(2\sqrt{d-1}\left(1-\alpha\frac{\log\log(n)}{\log(n)}\right)\right)^{2k} + 2m(2\sqrt{d-1})^{2k}. 
	\end{align*}
	If we let $\beta$ be sufficiently small and $\alpha>\frac3{2\beta}$, rearranging yields
	\begin{align*}
		\frac{m}n 
		   &\geq \frac{C\frac{4^k(d-1)^k}{k^{3/2}} -\frac{2d^{2k}}{n}-\left(2\sqrt{d-1}\left(1-\alpha\frac{\log\log(n)}{\log(n)}\right)\right)^{2k} }{2(2\sqrt{d-1})^{2k}\cdot \left(1-\left(1-\alpha\frac{\log\log(n)}{\log(n)}\right)^{2k}\right)}
		\\ &=  \Omega\left( \frac{1-2n^{2\beta-1}}{k^{3/2}} - \frac{\left(1-\alpha\frac{\log\log(n)}{\log(n)}\right)^{2k}}{1-\left(1-\alpha\frac{\log\log(n)}{\log(n)}\right)^{2k}}\right)
				\\ &=  \Omega\left( \frac{1}{k^{3/2}} - \frac{1}{e^{2\alpha\beta\log\log(n)}}\right)
		\\ &=\Omega\left( \frac{1}{k^{3/2}}\right).
	\end{align*}
\end{proof}

\subsection{Mangrove Tree}\label{sec:examples3}
This section shows that the dependence on $|V|$ in Corollary \ref{cor:irregperron1} is tight up to polylogarithmic factors. Our example begins with a path of multiedges containing $n$ vertices, where each multiedge of the path is composed of $d/2$ edges for some even $d$. At both ends of the path, we attach a tree of depth $\log_{d-1}n$. The roots have degree $d/2$ and all other vertices (besides the leaves) have degree $d$. Therefore the only vertices in the graph that are not degree $d$ are the leaves of the two trees. Call this graph $Q$. An example of this graph is shown in Figure \ref{fig:palm}.
\begin{prop}
For every vertex $u$ of degree less than $d$, $\psi_Q(u)=\tilde O(n^{-5/2})$, where $\tilde O$ suppresses dependence on logarithmic factors and $d$. 
\end{prop}
Therefore, we cannot hope to do significantly better than our analysis in Lemma \ref{lem:increase}, in which we find a vertex $u$ of non-maximal degree with $\psi(u)\geq1/(d\lambda_1n^{5/2})$.

\begin{proof}
For simplicity, call $\lambda_1(A_Q)=\lambda_1$ and $\psi_Q=\psi$. By the symmetry of the graph, the value of $\psi$ at vertices in the tree is determined by the distance from the root. Call the entries of $\psi$ corresponding to the tree $r_0,r_1,\ldots, r_\ell$, where the index indicates the distance from the root. 

By the discussion in the proof of Kahale \cite{Kah} Lemma 3.3, if we define $$\theta:=\log\left(\frac{\lambda_1}{2\sqrt{d-1}}+\sqrt{\frac{\lambda_1^2}{4(d-1)}-1}\right),$$ then for $0\leq i\leq \ell$, entries of the eigenvector must satisfy 
\[
\frac{r_i}{r_0}=\frac{\sinh((\ell+1-i)\theta)(d-1)^{-i/2}}{\sinh((\ell+1)\theta)}
\]
where $\ell$ is the depth of the tree.

Therefore, $r_\ell/r_0=\frac{\sinh(\theta)(d-1)^{-\ell/2}}{\sinh((\ell+1)\theta)}$.
Examining the various terms, $\sinh(\theta)\leq d$ and $(d-1)^{-\ell/2}=\frac{1}{\sqrt{n}}$. To bound the third term, we use the definition $\sinh(x)=(e^x-e^{-x})/2$, which yields
\[
\sinh((\ell+1)\theta)\geq \frac{1-o_n(1)}2\left(\frac{\lambda_1}{2\sqrt{d-1}}+\sqrt{\frac{\lambda_1^2}{4(d-1)}-1}\right)^{\log_{d-1}n+1}.
\]
$\lambda_1$ is at least the spectral radius of the path of length $n$ with $d/2$ multiedges between vertices. This spectral radius is $d\cos(\pi/(n+1))$.  This gives
\begin{eqnarray*}
\sinh((\ell+1)\theta)&\geq&  \frac{1-o_n(1)}{2(2\sqrt{d-1})^{\log_{d-1}n+1}}\left(d(1-\frac{\pi^2}{2n^2})+\sqrt{d^2(1-\frac{\pi^2}{2n^2})^2-4d+4}\right)^{\log_{d-1}n+1}\\
&\geq&\frac{1-o_n(1)}{2(2\sqrt{d-1})^{\log_{d-1}n+1}}\left(d+d-2\right)^{\log_{d-1}n+1}\left(1-O\left(\frac{d}{n^2}\right)\right)^{\log_{d-1}n+1}\\
&\geq&\frac{1-o_n(1)}{2}e^{-O({d\log_{d-1} n}/{n^2})}\sqrt{n} \geq\frac{\sqrt n}3
\end{eqnarray*}
for large enough $n$. Therefore
\begin{equation}\label{eq:endbound}
\frac{r_\ell}{r_0}=\frac{\sinh(\theta)(d-1)^{-\ell/2}}{\sinh((\ell+1)\theta)}\leq \frac{3d}{ n}.
\end{equation}

At this point, we know the ratio between $r_\ell$ and $r_0$, but still need to bound the overall mass of the eigenvector on the tree. A ``regular partition'' is a partition of vertices $V=\bigsqcup_{i=0}^k X_j$ where the number of neighbors a vertex $u\in X_i$ has in $X_j$ does not depend on $u$. We can create a \textit{quotient matrix}, where entry $i,j$ corresponds to the number of neighbors a vertex $u\in X_i$ has in $X_j$. For an overview of quotient matrices and their utility, see Godsil, \cite[Chapter 5]{Godsil}. In our partition, every vertex in the path is placed in a set by itself. The vertices of each of the two trees are partitioned into sets according to their distance from the two roots.  Call the matrix according to this partition $B_Q$. We denote by $B_Q(X_i,X_j)$ the entry in $B_Q$ corresponding to edges from a vertex in $X_i$ to $X_j$.

Define $X_0,\ldots X_\ell$ as the sets corresponding to vertices in the first tree of distance $0,\ldots, \ell$ from the root. For $1\leq j\leq \ell-1$, $B_Q(X_{0},X_{1})=d/2$. $B_Q(X_j,X_{j+1})=d-1$.  Moreover, for $0\leq j\leq \ell-1$, $B_Q(X_{j+1},X_{j})=1$. All values between vertices in the path are unchanged at $d/2$.

Consider the diagonal matrix $D$ with $D_{i,i}:=|X_i|^{-1/2}$. $D^{-1}B_Q D$ is a symmetric matrix. Define $C:=D^{-1}B_Q D$ We now have $C(X_{j+1},X_{j})=C(X_{j},X_{j+1})=\sqrt{d-1}$ for $1\leq j\leq \ell-1$, and $C(X_{0},X_{1})=C(X_{1},X_{0})=\sqrt{d/2}$.

If a vector $\phi$ is an eigenvector of $C$, then $D\phi$ is an eigenvector of $B_Q$ with the same eigenvalue. By the definition of $D$ this means
\begin{equation}\label{eq:mass}
\psi_C(X_i)^2=\sum_{u\in X_i}\psi_{A_Q}(u)^2.
\end{equation}

 Define $C_{X_{0:\ell}}$ as the submatrix of $C$ corresponding the the sets $\{X_0,\ldots,X_\ell\}$, then extended with zeros to have the same size as $C$. Every entry of $C+\frac{d/2-\sqrt{d-1}}{\sqrt{d-1}}C_{X_{0:\ell}}$ is less than or equal to the corresponding entry of the adjacency matrix of a path of length $n+2\log_{d-1}n$ with $d/2$ edges between pairs of vertices. Also, $\psi_C$ is a nonnegative vector. Therefore the quadratic form $\psi_C^T(C+\frac{d/2-\sqrt{d-1}}{\sqrt{d-1}}C_{X_{0:\ell}})\psi_C$ is at most the spectral radius of this path. Namely
\[
\psi_C^TC\psi_C+\frac{d/2-\sqrt{d-1}}{\sqrt{d-1}}\psi_C^TC_{X_{0:\ell}}\psi_C\leq d\cos(\pi/(n+2\log_{d-1}n+1)).
\]
Because $C$ contains the path of length $n$, $\psi_C^TC\psi_ C\geq d\cos(\pi/(n+1))$. Putting these together yields
\begin{equation}\label{eq:palmeq}
\psi_C^TC_{X_{0:\ell}}\psi_C\leq \frac{\sqrt{d-1}}{d/2-\sqrt{d-1}}\cdot d (\cos(\pi/(n+2\log_{d-1}n+1))-\cos(\pi/(n+1)))\leq \frac{d\sqrt{d-1}}{d/2-\sqrt{d-1}}\frac{3\pi^2\log_d n}{n^3}.
\end{equation}
Define $\psi_C(X_{1:\ell})$ as the projection of $\psi_C$ on $\{X_1,\ldots X_\ell\}$. $C\psi_C(X_{1:\ell})=C_{X_{0:\ell}}\psi_C(X_{1:\ell})$, so
\begin{equation}\label{eq:palmeq2}
\psi_C^TC_{X_{0:\ell}}\psi_C\geq \lambda_1\|\psi_C(X_{1:\ell})\|^2\geq d(\cos(\pi/(n+1)))\|\psi_C(X_{1:\ell})\|^2
\end{equation}
Combining (\ref{eq:palmeq}) and (\ref{eq:palmeq2}) yields 
\[
\|\psi_C(X_{1:\ell})\|^2\leq \frac{\sqrt{d-1}}{d/2-\sqrt{d-1}}\left(\frac{3\pi^2\log_d n}{n^3}\right)/\cos(\pi/n+1)\leq \left(\frac{21\pi^2\log_d n}{n^3}\right)
\]
assuming $d\geq 4$ and $n$ is sufficiently large.

Using $(\ref{eq:mass})$ and the eigenvalue equation, we obtain \[\psi_Q(r_0)=\psi_C(X_0)\leq\lambda_1(A_C)\|\psi_C(X_{1:\ell})\|\leq d\cdot \frac{{5}\pi\log^{1/2}_d n}{n^{3/2}}.\] Therefore, according to (\ref{eq:endbound})
\[
r_\ell\leq \frac{15d^2\pi(\log^{1/2}_d n)}{n^{5/2}}.
\]

\end{proof}

\section{Open Problems}\label{sec:openprobs}
We conclude with some promising directions for further research.

\subsection*{Beyond the Trace Method: Polynomial Multiplicity Bounds} There is a large gap between our upper bound of $O(n/\log^{1/5}n)$ on the multiplicity of the second eigenvalue and the lower bound of $n^{1/3}$ mentioned after \cref{thm:zhao}. It is very natural to ask, whether the bound of this paper may be improved. To improve the bound beyond $O(n/\polylog(n))$, however, it appears that a very different approach is needed.
\begin{openprob}[Similar to Question 6.3 of \cite{equiangular}]
Let $d>1$ be fixed integer. Does there exist an $\eps>0$ such that for every connected $d$-regular graph $G$ on $n$ vertices, the multiplicity of the second largest eigenvalue of $A_G$ is $O(n^{1-\eps})$?
\end{openprob}
In the present paper, we rely on the trace method to bound eigenvalue multiplicity through closed walks. There are three drawbacks to this approach that stops a bound on the second eigenvalue multiplicity below $n/\polylog(n)$. First, considering walks of length $\omega(\log(n))$ makes the top eigenvalue dominate the trace, leaving no information behind. Second, considering the trace $\tr A_G^k$ for $k=O(\log(n))$ it is impossible to distinguish eigenvalues that differ by $O(1/\log(n))$. Third, as covered in \cref{sec:examples1}, there exist graphs such that there are $\Omega(n/\polylog(n))$ eigenvalues in a range of that size around the second eigenvalue. Thus, the trace method reaches a natural barrier at $n/\polylog(n))$.

\subsection*{Eigenvalue Multiplicity for Unnormalized Non-Regular Graphs} Another natural question is whether \cref{thm:multbound1} may be extended to hold for the (non-normalized) adjacency matrix of non-regular graphs. 
\begin{openprob}
    Let $\Delta>1$ be a fixed integer. Does it hold for every connected graph $G$ on $n$ vertices of maximum degree $\Delta$ that the multiplicity of the second largest eigenvalue of $A_G$ is $o(n/\log\log(n))$?
\end{openprob}

 In order to handle unnormalized irregular graphs via the approach in this paper, the key ingredient needed would be an ``unnormalized'' analogue of Theorem \ref{thm:closed}, showing that a uniformly random closed walk (from the set of all closed walks) in an irregular graph must have large support. We exhibit in Appendix \ref{sec:examples2} an irregular ``lollipop'' graph for which the typical support of a closed walk from a specific vertex is only $O(\polylog(k))$. It remains plausible that when starting from a random vertex, a randomly selected closed walk has poly$(k)$ support in irregular graphs.

\subsection*{Sharper Bounds for Closed Random Walks} We have no reason to believe that the exponent of $1/5$ appearing in  \cref{thm:closed} is sharp. In fact, we know of no example where where the answer is $o(k^{1/2})$. An improvement over \cref{thm:closed} would immediately yield an improvement of \cref{thm:multbound1}. 
\begin{openprob}
    Let $d>1$ be a fixed integer. Does there exist an $\alpha> 1/5$ such that for every connected $d$-regular graph $G$ on $n$ vertices and every vertex $x$ of $G$, a random closed walk of length $2k<n$ rooted at $x$ has support $\Omega(k^\alpha)$ in expectation? Is $\alpha=1/2$ true? Does such a bound hold for SRW in general?
\end{openprob}

\subsection*{Acknowledgments}
We thank Yufei Zhao for telling us about \cite{equiangular} at the Simons Foundation conference on High Dimensional Expanders in October, 2019. We thank Shirshendu Ganguly for helpful discussions. We thank Cyril Letrouit for pointing out an error in the previous proof of Proposition 5.2.

\bibliographystyle{alpha}
\bibliography{ref}
\appendix
\section{Proofs for high degree regular graphs}\label{sec:highdegproofs}

\begin{theorem}[Detailed Theorem \ref{thm:highdegwalk}]\label{thm:highdegclosed}
If $G$ is $d$-regular, has exactly $h$ self-loops at every vertex, and no multi-edges\footnote{This technical assumption is used to handle the case when $|\lambda_n(A_G)|>\lambda_2(A_G)$ in Theorem \ref{thm:highdeg2}. Here we take $h=0$.}, then
\begin{equation}\label{eqn:hicycles}
\P_x(W^{2k,s})\leq \exptwo\P_x(W^{2k,2s}) \qquad\textrm{for}\quad s\le \stwo.
\end{equation}
\end{theorem}
\begin{proof}
We show this via a small modification of the proof of Theorem \ref{thm:cyclesup}. Assume $s\le (d-h)/2$. The key observation is that each vertex has at least $d-h$ edges in $G$ to other vertices, so in a  subgraph of size at most $2s-1$ every vertex has at least one edge in $G$ leaving the subgraph. In this case, we can simply choose $u\in S$ as $u:=\arg\max_{w\in S} \psi_S(w)$ in Lemma \ref{lem:increase}. Therefore, considering the adjacency matrix,  (\ref{eq:onevert}) can be improved to 

\begin{equation*}
\lambda_1(A_{S\cup\{v\}})\geq\frac{1}{2}\left(\lambda_1+\sqrt{\lambda_1^2+\psi_S(u)^2}\right)\geq \lambda_1+\frac{\psi_S(u)^2}{6\lambda_1^2}\geq \lambda_1+\frac{1}{6\lambda_1^2s}.
\end{equation*}

Therefore, after adding $s$ vertices to $S$ according to the process of Lemma \ref{lem:increase}, we find a set $T\in\Gamma_x^{2s}$ satisfying
\[
\lambda_1(A_T)\geq \lambda_1+\frac{1}{6\lambda_1^2}\sum_{i=1}^s \frac1{s+i-1}\geq\lambda_1+\frac{\log 2}{6\lambda_1^2}\ge \lambda_1\left(1+\frac{1}{10\lambda_1^3}\right).
\]

Using this improved bound, and keeping in mind that $\lambda_1(A_T)\leq 2s$, we can replicate the argument above to get to the following improvement over (\ref{eqn:wbound}):
\begin{align*}
\P_x(W^{2k,s}) &\leq \exp\left(2s\log d+4s\log (2s)+\log (2s)-\frac{k}{80s^3}\right)\P_x(W^{2k,2s}).
\end{align*}
This implies
\begin{align*}
\P_x(W^{2k,s})\leq \exp\left(7s\log d-\frac{k}{80s^3}\right)\P_x(W^{2k,2s})
\leq \exp\left(-\frac{k}{100s^3}\right)\P_x(W_x^{2k,2s})
\end{align*}
whenever
$$ s\le \frac{1}{8}\left(\frac{k}{\log d}\right)^{1/4},$$
establishing \eqref{eqn:hicycles}.
\end{proof}

\begin{theorem}[Detailed Theorem \ref{thm:highdegmult}]\label{thm:highdeg2}

If $G$ is simple and $d$-regular, then
\begin{equation}\label{eqn:highdeg}
    m_G\left([(1-\frac{\log\log_d n}{\log_d n})\lambda_2,\lambda_2]\right)=
    \begin{cases}
    \O\left(n\cdot \frac{\log d\log\log n}{d}\right) & \textrm{ when } d\log^{1/2}d\le \alpha\log^{1/4} n\\
    \O\left(n\cdot \frac{\log^{1/2} d\log\log n}{\log^{1/4} n}\right) &\textrm{ when }
    d\log^{1/2}d\ge \alpha\log^{1/4}n
    \end{cases}
\end{equation}
for all\footnote{If $d\ge \exp(\sqrt{\log n})$ then \eqref{eqn:highdeg} is vacuously true.} $d\le \exp(\sqrt{\log n})$, where $\alpha:=\sqrt[4]{3}/4$.
\end{theorem}

\begin{proof}

The proof is the same as the proof of Theorem \ref{thm:multbound1} in Section \ref{sec:mult}, except we choose different $s$. 
\begin{enumerate}
\item
If
$d\log^{1/2}d< \alpha\log^{1/4}n$ set
$$ s:= \stwo = \frac{d}{2}$$
with $h=0$.
Applying Theorem \ref{thm:highdegclosed} it is easily checked that \eqref{eqn:sets} is satisfied for large enough $n$, yielding a bound of
$$m_G\left([(1-\frac{\log\log_d n}{\log_d n})\lambda_2,\lambda_2]\right)=O\left(n\cdot \frac{\log d\log\log n}{d}\right).$$

\item If $G$ is simple, $d$-regular and $d\log^{1/2}d\ge \alpha\log^{1/4} n$, set 
$$ s:= \stwo = \frac{1}{8}\left(\frac{\log n}{\log^2 d}\right)^{1/4}$$
with $h=0$.
Then \eqref{eqn:sets} is again satisfied by applying \cref{thm:cyclesup} equation \eqref{eqn:hicycles}, and we conclude that
$$m_G\left([(1-\frac{\log\log_d n}{\log_d n})\lambda_2,\lambda_2]\right)=O\left(n\cdot \frac{\log^{1/2} d\log\log n}{\log^{1/4} n}\right).$$
\end{enumerate}
\end{proof}

\section{Lollipop}\label{sec:examples2}
Here, we show that if we do not assume that our graph is regular, the average support of a uniformly chosen (from the set of all such walks)  closed walk of length $k$ from a fixed vertex is no longer necessarily $k^{\Theta(1)}$ (as opposed to the average support of a random walk) . We take the lollipop graph, which consists of a clique of  $(d+1)$ vertices for fixed $d\geq 3$ and a path of length $n$ $\{u_1,\ldots,u_n\}$ attached to a vertex $v$ of the clique, where $n\gg k$. Here $\psi:=\psi(A)$ and $\lambda_1:=\lambda_1(A)$ are the Perron eigenvector and eigenvalue of the adjacency matrix of the graph.
\begin{lemma}\label{lem:ip}
$\psi(v)\geq 1/\sqrt {d+2}$.
\end{lemma}
\begin{proof}
By symmetry, the value on all entries of the clique besides $v$ are the same. Call this value $\psi(b)$. Then by the eigenvalue equation we have
$\lambda_1 \psi(b)=\psi(v)+(d-1)\psi(b)$, so as $\lambda_1\geq d$, it must be that $\psi(v)\geq \psi(b)$. 

Similarly, to satisfy the eigenvalue equation, vertices on the path must satisfy the recursive relation
\[
\begin{array}{cc}
\lambda_1 \psi(u_i)=\psi(u_{i-1})+\psi(u_{i+1}) & 1\leq i\leq n-1\\
\lambda_1 \psi(u_n)=\psi(u_{n-1})
\end{array}
\]
where we define $v=u_0$. To satisfy this equation, we must have $\psi(u_i)\geq (\lambda_1-1)\psi(u_{i+1})$ for each $i$, so as $\lambda_1\geq d\geq 3$, $\psi(v)\geq \sum_{i=1}^n \psi(u_k)$. As the Perron vector is nonnegative, $\psi(v)^2\geq \sum_{i=1}^n \psi(u_k)^2$, and 
\[(d+2)\psi(v)^2\geq \psi(v)^2+d \psi(b)^2 +\sum_{i=1}^n \psi(u_k)^2=1,\]
so $\psi(v)\geq 1/\sqrt{d+2}$.
\end{proof}
Call $\gamma_v^{2k}$ the number of closed walks of length $2k$, and $\gamma_v^{2k,\geq \ell+d+1}$ as the subset of these walks with support at least $\ell+d+1$.
\begin{prop}
For $\ell\geq 2\log(k)/\log(\lambda_1/2)$,
 \[|\gamma_v^{2k,\geq \ell+d+1}|=O(k^{-2})|\gamma_v^{2k}|.\]
\end{prop}

\begin{proof}

For a closed walk to have support $\ell+d+1$, it must contain $u_\ell$. For such walks, once the path is entered, at least $2\ell$ steps must be spent in the path, as the walk must reach $u_{\ell}$ and return. Therefore, closed walks starting at $v$ that reach $u_\ell$ can be categorized as follows. First, there is a closed walk from $v$ to $v$. Then there is a closed walk from $v$ to $v$ going down the path containing $u_\ell$. On this excursion, the walk can only go forward or backwards, and it spends at least $2\ell$ steps within the path. For each of these steps, there are $2$ options. If we remain in the path after $2\ell$ steps, upper bound the number of choices until returning to $v$ by $\lambda_1$ at each step. After returning to $v$, the remaining steps form another closed walk. The number of closed walks from $v$ of length $i$ is at most $\lambda_1^i$. Therefore the number of closed walks with an excursion to $u_\ell$ is at most 
\[
\sum_{i=0}^{2k} \lambda_1^i 2^{2\ell}\lambda_1^{2k-2\ell -i}=(2k+1)\lambda_1^{2k-2\ell}2^{2\ell}.
\]

The total number of closed walks starting at $v$ is at least $\psi(v)^2\lambda_1^n$. Therefore the fraction of closed walks that have support at least $\ell$ is at most

\[
\frac{(2k+1)2^{2\ell} \lambda_1^{2k-2\ell}}{\lambda_1^{2k}/(d+2)}=\frac{(d+2)(2k+1)2^{2\ell}}{\lambda_1^{2\ell}}
\]
so for $\ell\geq 2(\log k)/\log(\lambda_1/2)$, this is $O(k^{-2})$.

\end{proof}
\begin{remark}
Instead of adding a path, we can add a tree (as exhibited in Figure \ref{fig:biglolli}). According to the same analysis, the probability a walk reaches depth further than $\Theta(\log k)$ is small. Therefore, in Theorem \ref{thm:closed} we can not get a sufficient bound on support from passing to depth, but must deal with support itself.
\end{remark}

\end{document}